\documentclass{amsproc}
\usepackage{amsmath,amsfonts,textcomp,longtable,pstricks}
\usepackage[dvips]{graphicx}

\newtheorem{thm}{Theorem }[section]

\newtheorem{lemma}[thm]{Lemma }

\newtheorem{prop}[thm]{Proposition }

\theoremstyle{definition}
\newtheorem{deff}[thm]{Definition }

\numberwithin{equation}{section}

\def\RR{{\mathbb R}}

\def\QQ{{\mathbb Q}}
\def\ZZ{{\mathbb Z}}

\def\fq{{\mathbb{F}_q}}
\def\kk{{\bar{k}}}

\def \bra#1\ket {\mathop{\vphantom{#1}\left<\smash{#1}\right>}\nolimits}

 \DeclareMathOperator{\im}{Im}
 
\DeclareMathOperator{\End}{End}

 \DeclareMathOperator{\rk}{rk}
 \DeclareMathOperator{\diag}{diag}

 \DeclareMathOperator{\Np}{Np}
\DeclareMathOperator{\Hp}{Hp} 
 \DeclareMathOperator{\ord}{ord}

\def\plim{\mathop{{\lim\limits_{\longleftarrow}}}\nolimits}

\renewcommand \phi {\varphi}
\renewcommand \rho {\varrho}

\begin{document}
\author{Sergey Rybakov}
\thanks{Supported in part by RFBR grants no. 11-01-00393-a, 11-01-12072-ofi-m and 10-01-93110-CNRSLa}
\address{Poncelet laboratory (UMI 2615 of CNRS and Independent University of
Moscow)}
\address{Institute for information transmission problems of the Russian Academy of Sciences}
\address{Laboratory of Algebraic Geometry, GU-HSE, 7 Vavilova Str., Moscow, Russia, 117312}

\email{rybakov@mccme.ru, rybakov.sergey@gmail.com}%

\title[The groups of points on abelian surfaces]
{The groups of points on abelian surfaces\\ over finite fields}
\date{}


\begin{abstract}
Let $A$ be an abelian surface over a finite field $k$. The
$k$-isogeny class of $A$ is uniquely determined by a Weil
polynomial $f_A$ of degree $4$. We give a classification of the
groups of $k$-rational points on varieties from this class in
terms of $f_A$.
\end{abstract}

\maketitle
\section{Introduction}
Classical results of Tate and Honda give us an explicit
description for the set of isogeny classes of abelian varieties
over a finite field $k$. Namely, an isogeny class is uniquely
determined by a so called Weil polynomial of any variety from
this class. It looks natural to consider classification problems
concerning abelian varieties inside a given isogeny class. In this
paper we classify groups of $k$-rational points on abelian
surfaces. Xing partly classified groups of points on supersingular 
surfaces in~\cite{Xi1} and in~\cite{Xi2}. In the paper~\cite{Ry1} we show 
that one could use the language of Hodge polygons to describe groups of points. Moreover, we 
classify this groups for abelian varieties with commutative endomorphism algebra. 
An abelian variety has commutative endomorphism algebra if and only if its Weil polynomial has no multiple roots.
We could say that this is the general case. From the classification of Weil polynomials of abelian surfaces due to
R\"uck, Xing, Maisner and Nart~\cite{MN} we get that there are only three cases more. Namely,
\begin{itemize}
	\item the Weil polynomial is a square of a polynomial of degree $2$ without multiple roots;
	\item the Weil polynomial is of the form $P(t)(t\pm\sqrt{q})^2$, where $P(t)$ has no multiple roots, and $\sqrt{q}\in\ZZ$;
	\item the Weil polynomial equals $(t\pm\sqrt{q})^4$, and $\sqrt{q}\in\ZZ$.
\end{itemize}
 
The paper is devoted to a classification of groups of points for these three cases.

The author is grateful to referee for providing useful corrections on the paper.

\section{Main result}
Throughout this paper $k$ is a finite field $\fq$ of
characteristic $p$. Let $A$ be an abelian variety of dimension $g$
over $k$, and let $\kk$ be an algebraic closure of $k$. Fix a
prime number $\ell$. For a natural number $m$ denote by $A_m$ the
kernel of multiplication by $\ell^m$ in $A(\kk)$. Let $T_\ell(A) =
\plim A_m$ be the Tate module, and
$V_\ell(A)=T_\ell(A)\otimes_{\ZZ_\ell}\QQ_\ell$ be the
corresponding vector space over $\QQ_\ell$. If $\ell\ne p$, then
$T_\ell(A)$ is a free $\ZZ_\ell$-module of rank $2g$. The
Frobenius endomorphism $F$ of $A$ acts on the Tate module by a
semisimple linear operator, which we also denote by $F:
T_\ell(A)\to T_\ell(A)$. The characteristic polynomial
$$
f_A(t) = \det(t-F|T_\ell(A))
$$
is called a {\it Weil polynomial of $A$}. It is a monic polynomial
of degree $2g$ with rational integer coefficients independent of
the choice of prime $\ell\ne p$. It is well known that for
isogenous varieties $A$ and $B$ we have $f_A(t)=f_B(t)$. Moreover,
Tate proved that the isogeny class of abelian variety is
determined by its characteristic polynomial, that is
$f_A(t)=f_B(t)$ implies that $A$ is isogenous to $B$ (see
\cite{WM}). If $\ell=p$, then $T_p(A)$ is called a {\it physical
Tate module}. In this case, $f_A(t)=f_1(t)f_2(t)$, where
$f_1,f_2\in\ZZ_p[t]$, and $f_1(t)=\det(t-F|T_p(A))$. Moreover
$d=\deg f_1\leq g$, and $f_2(t)\equiv t^{2g-d}\mod p$
(see~\cite{De}).

Recall some definitions and results from~\cite{Ry1}. For an
abelian group $G$ we denote by $G_\ell$ the $\ell$-primary
component of $G$. The group $A(k)$ is a kernel of $1-F: A\to A$, and the $\ell$-component
$$A(k)_\ell\cong T_\ell(A)/(1-F)T_\ell(A)$$ (see~\cite[Proposition 3.1]{Ry1}). The proof of
the following lemma is essentially the proof of~\cite[Theorem~1.1]{Ry1}.

\begin{lemma}
\label{lem_on_Tate_module} 
Let $A$ be an abelian variety over $k$, and let $G$ be a finite abelian group of order $f_A(1)$. Suppose that for any prime number $\ell$ dividing order of $G$ we have an $F$-invariant sublattice $T_\ell\subset T_\ell(A)$ such that $G_\ell\cong T_\ell/(1-F)T_\ell$. Then there exists an abelian variety $B$ over $k$, and an isogeny $\phi:B\to A$ such that $T_\ell(\phi)$ induces an isomorphism $T_\ell(B)\cong T_\ell$ for any $\ell$. In particular, $B(k)\cong G$. 
\end{lemma}

We need a modification of this statement. Suppose we are looking for an abelian variety $B$ such that $B(k)\cong G$, and $A$ and $B$ are isogenous. It is enough to find free $\ZZ_\ell$ module $T_\ell$ for any $\ell$ with semisimple action $F:T_\ell\to T_\ell$ such that $f_A(t)=\det(t-F)$, and $G_\ell\cong T_\ell/(1-F)T_\ell$. Indeed, since the Frobenius action on the vector space $V_\ell$ is semisimple, it is determined up to isomorphism by the Weil polynomial $f_A$. We can reconstruct this action as follows. Put $V_\ell'=T_\ell\otimes_{\ZZ_\ell}\QQ_\ell$, and choose any inclusion $T_\ell(A)\to V_\ell'$ such that the image of $T_\ell(A)$ contains $T_\ell$. Clearly, this inclusion induces an isomorphism of $F$-vector spaces $V_\ell\to V_\ell'$. 

Let $Q(t)=\sum_i Q_i t^i$ be a polynomial of degree $d$ over
$\QQ_\ell$, and let $Q(0)=Q_0\neq 0$. Take the lower convex hull
of the points $(i,\ord_\ell(Q_i))$ for $0\leq i\leq d$ in $\RR^2$.
The boundary of this region without vertical lines is called {\it
the Newton polygon $\Np_\ell(Q)$ of $Q$}. Its vertices have
integer coefficients, and $(0,\ord_\ell(Q_0))$ and $(d,\ord_\ell(Q_d))$ are its endpoints. One can associate to $\Np(Q)$ the set of its slopes, and each slope has a multiplicity. The Newton polygon of $Q$ is uniquely determined by this data. Suppose we have two polynomials $Q_1$ and $Q_2$. Then the slope set for $\Np(Q_1Q_2)$ is the union of slope sets for $\Np(Q_1)$ and $\Np(Q_2)$, and multiplicities of slopes of $\Np(Q_1Q_2)$ are sums of multiplicities of corresponding slopes of $\Np(Q_1)$ and $\Np(Q_2)$.

We associate to $A(k)_\ell$ a polygon of special type.

\begin{deff}Let $0\leq m_1\leq m_2\leq\dots\leq m_r$ be nonnegative
integers, and let $H=\oplus_{i=1}^r\ZZ/\ell^{m_i}\ZZ$ be an
abelian group of order $\ell^m$. The {\it Hodge polygon
$\Hp_\ell(H,r)$ of a group $H$} is the convex polygon with
vertices $(i,\sum_{j=1}^{r-i}m_j)$ for $0\leq i\leq r$. It has
$(0,m)$ and $(r,0)$ as its endpoints, and its slopes are
$-m_r,\dots, -m_1$. We write $\Hp(H)=(m_1,\dots,m_r)$.
\end{deff}

Note that some of the numbers $m_i$ could be zero, in other words, the Hodge polygon is allowed to have zero slopes. The isomorphism class of $H$ depends only on $\Hp_\ell(H,r)$.
When we work with groups of points on abelian surfaces we often
write $\Hp_\ell(H)=\Hp_\ell(H,4)$. We need the following well-known result (see~\cite[4.3.8]{Ke2008} or \cite[8.40]{BO}).

\begin{thm}\label{rp}
Let $E$ be an injective endomorphism of a free $\ZZ_\ell$-module $T$
of finite rank. Let $f(t)=\det(E-t)$ be its characteristic
polynomial. Then $\Np_\ell(f)$ lies on or above the Hodge polygon
of $T/ET$, and these polygons have same endpoints.
\end{thm}

Weil polynomials of abelian surfaces were classified by R\"uck,
Xing, Maisner and Nart~\cite{MN}. We use a simplified version of
this classification. Namely, we use that for the Weil polynomial
only four cases of the main theorem below occur.

\begin{thm}\label{2dim}
Let $A$ be an abelian surface over a finite field with Weil
polynomial $f_A$. Let $G$ be an abelian group of order $f_A(1)$.
Then $G$ is a group of points on some variety in the isogeny class
of $A$ if and only if for any prime number $\ell$
\begin{enumerate}
\item if $f_A$ has no multiple roots, then
$\Np_\ell(f_A(1-t))$ lies on or above $\Hp_\ell(G_\ell)$;
\item if $f_A=P_A^2$, and $P_A$ has no
multiple roots, then $G_\ell\cong G_\ell^{(1)}\oplus G_\ell^{(2)}$, where $G_\ell^{(1)}$ and
$G_\ell^{(2)}$ are $\ell$-primary abelian groups with one or two generators
such that $\Np_\ell(P_A(1-t))$ lies on or above $\Hp_\ell(G_\ell^{(1)},2)$
and $\Hp_\ell(G_\ell^{(2)},2)$.
\item Suppose $f_A=(t^2-bt+q)(t\pm\sqrt{q})^2$, and $\sqrt{q}\in\ZZ$, where $f(t)=t^2-bt+q$ has no multiple
roots; let
\begin{itemize}
\item $\Hp_\ell(G_\ell)=(m_1,m_2,m_3,m_4)$;
\item $m=\ord_\ell(f(1))$;
\item $m_q=\ord_\ell(1\pm\sqrt{q})$;
\item $m_b=m_1+m_3-m_q$.
\end{itemize}
Then
\begin{enumerate}
\item $0\leq m_b\leq\ord_\ell(b-2)$;
\item $\min(m_b,m_q)\geq m_1$;
\item $\min(m-m_b,m_q)\geq m_2$.
\end{enumerate}
\item Finally, if $f_A(t)=(t\pm\sqrt{q})^4$, and $\sqrt{q}\in\ZZ$, then
$G\cong(\ZZ/(1\pm\sqrt{q})\ZZ)^4$.
\end{enumerate}
\end{thm}
\begin{proof}
Case $(1)$ follows from~\cite{Ry1}. Let us prove that
conditions of the case $(2)$ are necessary. First, suppose that a
prime $\ell$ divides $f_A(1)$, but $P_A(1-t)\not\equiv
t^2\mod\ell.$ It is equivalent to say that
$P_A(1-t)=(t-x_1)(t-x_2)$, where $x_1,x_2\in\ZZ_\ell$, and
$\ell$ divides $x_1$, but not divides $x_2$. Thus
$T_\ell(A)$ is a direct sum of two modules $T_1$ and $T_2$ of
$\ZZ_\ell$-rank two such that $1-F$ acts on $T_i$ as
multiplication by $x_i$. It follows that $$G_\ell\cong
T_1/(1-F)T_1\cong(\ZZ_\ell/x_1\ZZ_\ell)^2, $$ and $G_\ell^{(1)}\cong G_\ell^{(2)}\cong\ZZ_\ell/x_1\ZZ_\ell$.
The rest follows from Theorem~\ref{case2} below. To prove that the conditions of $(2)$ are sufficient
we have to find an abelian variety $B$ with a given group of
points $G$. By Lemma~\ref{lem_on_Tate_module} it is enough to
construct a Tate module $T_\ell(B)$ for any prime $\ell$ dividing
$f_A(1)$. By \cite[Theorem 3.2]{Ry1} there exist $F$-invariant
$\ZZ_\ell$-modules $T'_1$ and $T'_2$ such that $T'_i/(1-F)T'_i\cong
G_\ell^{(i)}$, and $T'_i\otimes_{\ZZ_\ell}\QQ_\ell\cong\QQ_\ell[t]/P_A(t)\QQ_\ell[t]$
for $i\in\{1,2\}$. Now we let $T_\ell(B)\cong T'_1\oplus T'_2$.

We prove the case $(3)$ in the last section. The case $(4)$ is
obvious since $F$ acts as multiplication by $\mp\sqrt{q}$.
\end{proof}

\section{Matrix factorizations}
In this section we finish the proof of case (2). We assume that $\ell$ divides $P_A(1)$ and that
$$P_A(1-t)\equiv t^2\mod\ell.$$ In this case $\ell\neq p$. Indeed, let $P_A(t)=t^2-bt+c$, where $c=\pm q$.
If $\ell=p$ satisfies our assumptions, we get that $2-b$ and $1-b$ are both divisible by $p$.

In~\cite{Ry2} we show that the Tate module $T_\ell(A)$ corresponds to a matrix factorization. 
Fix a pair of polynomials $f, f_1\in \ZZ_\ell[t]$ and a
positive integer $r$. Put $R=\ZZ_\ell[t]/f_1\ZZ_\ell[t]$. Denote by $x\in R$ the
image of $t$ under the natural projection from $\ZZ_\ell[t]$.

\begin{deff}
A {matrix factorization} $(X,Y)$ is a pair of $r\times r$ matrices
with coefficients in $\ZZ_\ell[t]$ such that $\det
X=f$ and $YX=f_1\cdot I_r$, where $I_r$ is the identity matrix.
\end{deff}

Suppose we are given a matrix factorization $(X,Y)$.
The matrix $X$ defines a map of free $\ZZ_\ell[t]$ modules:
$$\ZZ_\ell[t]^r \xrightarrow{X} \ZZ_\ell[t]^r.$$ 
Its cokernel $T$ is annihilated by $f_1$. It is equivalent to say that $T$ is an $R$-module.
We see that the matrix factorization $(X,Y)$ corresponds to a finitely generated
$R$-module $T$ given by the presentation:
\begin{equation}\label{eq_mf}
\ZZ_\ell[t]^r \xrightarrow{X} \ZZ_\ell[t]^r \to T \to 0.
\end{equation}

\begin{prop}\cite[Proposition 5.2]{Ry2}\label{prop_mf1}
Suppose $f_1\equiv t^{d_1}\mod\ell$, and $\deg f_1=d_1$. The
module $T$ is free of finite rank over $\ZZ_\ell$, and
characteristic polynomial of the action of $x$ on $T$ is equal to
$f$.
\end{prop}

The following proposition shows that modules over $R$ give rise to
matrix factorizations.

\begin{prop}\label{prop_mf2}
Let $T$ be an $R$-module which is free of finite rank over $\ZZ_\ell$.
Suppose that $T$ can be generated over $R$ by $r$ elements, and
that $\Hp(T/xT)=(m_1,\dots,m_r)$. Then there exists a matrix
factorization $(X,Y)$ such that $T$ has presentation
${\rm(\ref{eq_mf})}$, and $$X\equiv
\diag(\ell^{m_1},\dots,\ell^{m_r})\mod t\ZZ_\ell[t].$$
\end{prop}
\begin{proof}
By \cite[Proposition 5.3]{Ry2} there exists a matrix factorization
$(X_1,Y_1)$ such that $T$ has presentation
\begin{equation}\label{eq_mf2}
\ZZ_\ell[t]^r \xrightarrow{X_1} \ZZ_\ell[t]^r \to T \to 0.
\end{equation}
Take the cokernel of the multiplication by $t$ of the
presentation${\rm(\ref{eq_mf2})}$; we get
$$\ZZ_\ell^r \xrightarrow{\bar X_1} \ZZ_\ell^r \to T/xT\to 0.$$

There exist matrices $M_1$ and $M_2$ over $\ZZ_\ell$ such that $\det M_1=\det M_2=1$
and $\bar X=M_1\bar X_1M_2$ is a diagonal matrix. Since $$\ZZ_\ell^r/\bar
X\ZZ_\ell^r\cong T/xT,$$ we get that $\bar
X=\diag(\ell^{m_1},\dots,\ell^{m_r})$. Now take $X=M_1X_1M_2$ and
$Y=M_2^{-1}Y_1M_1^{-1}$.
\end{proof}

\begin{thm}\label{case2}
Assume that $f_A=P_A^2$, and $P_A(1-t)\equiv t^2\mod\ell$. Suppose
that $P_A$ has no multiple roots. Then $A(k)_\ell\cong G_\ell^{(1)}\oplus
G_\ell^{(2)}$, where $G_\ell^{(1)}$ and $G_\ell^{(2)}$ are $\ell$-primary abelian groups with
one or two generators such that $\Np_\ell(P_A(1-t))$ lies on or
above $\Hp_\ell(G_\ell^{(1)},2)$ and $\Hp_\ell(G_\ell^{(2)},2)$.
\end{thm}
\begin{proof}
Apply Theorem~\ref{rp} to $T=T_\ell(A)$ and $E=1-F$. The slopes of $\Np(P_A(1-t))$ are the slopes of $\Np(P_A(1-t)^2)$, but multiplicities are doubled. We get that $m_1$ is not greater than the smallest slope of $\Np(P_A(1-t))$, i.e. that $\ell^{m_1}$ divides $b-2$.
It is enough to prove that $m:=\ord_\ell(P_A(1))=m_1+m_4$. Indeed, in this case $\Np_\ell(P_A(1-t))$ lies on or above $\Hp_\ell(G_\ell^{(1)},2)=(m_1,m_4)$. Moreover, $m=m_2+m_3$, and by Lemma~\ref{lemma} below $\Np_\ell(P_A(1-t))$ lies on or above $\Hp_\ell(G_\ell^{(2)},2)=(m_2,m_3)$.

Note that $m_4$ is a minimal number such that
$\frac{\ell^{m_4}}{F-1}\in\End A$, and $m_1$ is a maximal
number such that $\ell^{m_1}$ divides $F-1$. From the equality
$(F-1)^2-(b-2)(F-1)-P_A(1)=0$ we get that
$$\frac{P_A(1)}{(F-1)\ell^{m_1}}=\frac{F-1}{\ell^{m_1}}-
\frac{b-2}{\ell^{m_1}}\in\End A.$$ It follows that $m\geq
m_1+m_4$.

The Frobenius action on the Tate module $T_\ell(A)$ is determined
up to isomorphism by a module structure over the ring $R=\ZZ_\ell[t]/P_A(1-t)\ZZ_\ell[t]$
with $t$ acting as $1-F$. Suppose we have an $R$-module $T$ such
that $\Hp(T/xT)=(m_1,m_2,m_3,m_4)$. By Proposition~\ref{prop_mf2}
there exists a matrix factorization $(X,Y)$ such that $\det
X=f_A(1-t)$ and
$$X\equiv \diag(\ell^{m_1},\ell^{m_2},\ell^{m_3},\ell^{m_4})\mod t\ZZ_\ell[t].$$
The matrix factorization $(Y,X)$ corresponds to a module $T'$ over
$R$, which is generated by $4$ elements and the characteristic
polynomial of $x$ is equal to $\det
Y=P_A^4(1-t)/f_A(1-t)=f_A(1-t)$. Moreover,
$$Y\equiv \diag(\ell^{m-m_1},\ell^{m-m_2},\ell^{m-m_3},\ell^{m-m_4})\mod t\ZZ_\ell[t],$$
i.e. $\Hp(T'/xT')=(m-m_1,m-m_2,m-m_3,m-m_4)$. By
Lemma~\ref{lem_on_Tate_module} there exists an abelian surface $B$
such that $T_\ell(B)\cong T'$. From the first part of the proof
applied to $T_\ell(B)$ it follows that $(m-m_1)+(m-m_4)\leq m$.
Thus $m=m_1+m_4$.
\end{proof}

\begin{lemma}\label{lemma}
$m_2\leq\ord_\ell(b-2)$.
\end{lemma}
\begin{proof}
Let $L$ be a splitting field of $P_A$ over $\QQ_\ell$, and let $S$
be its ring of integers. We have $P_A(1-t)=(t-a)(t-c)$, where
$a,c\in S$, and $\ord_\ell(a)\leq\ord_\ell(b-2)$. (Here
$\ord_\ell(a)\in \frac{1}{2}\ZZ$.) Consider the map
$$E=F+a-1:T_\ell(A)\otimes_{\ZZ_\ell} S\to
T_\ell(A)\otimes_{\ZZ_\ell} S.$$ Then $1-F$ acts on $T_1=\ker E$
as multiplication by $a$. We get an exact sequence:
$$0\to T_1\to T_\ell(A)\otimes_{\ZZ_\ell} S \to T_2=\im E\to 0,$$
where $\rk_S T_1=\rk_S T_2=2$. Act by $1-F$ on this sequence. By snake lemma we get an exact
triple of abelian groups:
$$0\to (S/aS)^2\to A(k)_\ell\otimes_{\ZZ_\ell} S\to G'\to 0.$$

If $S=\ZZ_\ell$ then $A(k)_\ell$ contains
$(\ZZ_\ell/a\ZZ_\ell)^2$, and $G'=T_2/(1-F)T_2$ is generated by
$2$ elements. It follows that $m_2\leq\ord_\ell(a)$. If
$S\neq\ZZ_\ell$, then $S\cong\ZZ_\ell^2$ as $\ZZ_\ell$-module and
we multiply everything by two: $A(k)_\ell^2$ contains
$$(S/aS)^2\cong(\ZZ/\lfloor\ord_\ell(a)\rfloor\ZZ)^2\oplus
(\ZZ/\lceil\ord_\ell(a)\rceil\ZZ)^2,$$ and $G'=T_2/(1-F)T_2$ is generated by
$4$ elements. Thus
$m_2\leq\lfloor\ord_\ell(a)\rfloor\leq\ord_\ell(b-2)$.
\end{proof}

\section{Proof of case 3}
The conditions $(b)$ and $(c)$ are equivalent to the following
inequalities:
\begin{enumerate}
\item $m_1\leq m_b$, which is equivalent to $m_3\geq m_q$;
\item $m_1\leq m_q$;
\item $m_2\leq m_q$;
\item $m_2\leq m-m_b$, which is equivalent to $m_4\geq m_q$, since $f_A(1)=m_1+m_2+m_3+m_4=m+2m_q$.
\end{enumerate}
Note that inequalities $m_b\geq 0$ and $(4)$ follow from $(1)$,
and $(2)$ follows from $(3)$. We have to prove $(3)$,
$(1)$ and the second part of $(a)$.


Let $\alpha=1\pm\sqrt{q}$, and let $G_\ell=A(k)_\ell$. Consider
the map $E=F+\alpha-1:T_\ell(A)\to T_\ell(A)$. Then $1-F$ acts on
$T_1=\ker E$ by multiplication by $\alpha$, and $1-F$ acts on
$T_2=\im E$ with characteristic polynomial $f(1-t)$. We
get an exact sequence:
$$0\to T_1\to T_\ell(A)\to T_2\to 0,$$ 
and by the snake lemma
$$0\to (\ZZ/\ell^{m_q}\ZZ)^2\to
G_\ell\to G'\to 0.$$ The group $G_\ell$ contains
$(\ZZ/\ell^{m_q}\ZZ)^2$ only if $m_3\geq m_q$, and $G'$ has two
generators only if $(3)$ holds.

We now prove $(a)$. Since $\Hp_\ell(T_\ell(A)/(F-1)T_\ell(A))=(m_1,m_2,m_3,m_4)$ there exists
a basis $v_1,v_2,v_3,v_4\in T_\ell(A)$ such that $$(F-1)v_1\in\ell^{m_1}T_\ell(A)\text{,
}(F-1)v_2\in\ell^{m_2}T_\ell(A),$$
$$(F-1)v_3\in\ell^{m_3}T_\ell(A)\text{, and
}(F-1)v_4\in\ell^{m_4}T_\ell(A).\eqno (*)$$ The vectors $Ev_4$ and $Ev_3$ generate
a lattice $T_3\subset T_2$ of rank $2$, thus $1-F$ acts on $T_3$ with characteristic polynomial $f(1-t)$. 
By Theorem~\ref{rp} the Newton polygon $\Np_\ell(f(1-t))$ lies on or above $\Hp(T_3/(F-1)T_3,2)$. Thus, we
can find a linear combination $v_3'$ of $v_4$ and $v_3$ such that 
$v_1,v_2,v_3',v_4$ is a basis and $(*)$ holds, but
$(F-1)Ev_3'$ is not divisible by $(b-2)\ell$ in $T_3$.
On the other hand, $F-1=E-\alpha,$ and since $m_3\geq m_q$, it
follows that $Ev_3'$ is not divisible by $\alpha\ell$
in $T_\ell(A)$. We proved that $(F-1)Ev_3'$ is not divisible by
$\alpha (b-2)\ell$ in $T_\ell(A)$. By definition, $(F-1)v_3'$ is
divisible by $\ell^{m_3}$ in $T_\ell(A)$. By the first inequality,
$m_q\geq m_1$; thus for any vector $v\in T_\ell(A)$ the vector
$Ev$ is divisible by $\ell^{m_1}$ in $T_\ell(A)$. We finally get that $(F-1)Ev_3'$ is divisible by
$\ell^{m_1+m_3}$, but not divisible by $\ell\alpha(b-2)$ in $T_\ell(A)$. Thus
$$m_1+m_3<m_q+\ord_\ell(2-b)+1.$$ 

Let us construct an abelian variety with a given group of points.
We give an explicit construction of the Tate module
$T_\ell$ such that $T_\ell/(F-1)T_\ell\cong G_\ell$, and apply Lemma~\ref{lem_on_Tate_module}. 
By assumption, we have an isomorphism of $F$-vector spaces: $$V_\ell(A)\cong\QQ_\ell[t]/P_A(1-t)\oplus(\QQ_\ell[t]/(\alpha-t))^2,$$ where $t$ acts as $1-F$ on the right-hand side.
Thus there exists a basis $v_1,v_2,v_3,v_4$ of $V_\ell(A)$ such that
$$(1-F)v_1=\ell^{m_b}v_2+(b-2)v_1,\quad (1-F)v_2=-\ell^{-m_b}f(1)v_1,$$
$$(1-F)v_3=\alpha v_3,\quad\quad(1-F)v_4=\alpha v_4.$$ 

Let $T_\ell$ be the $\ZZ_\ell$-submodule of $V_\ell(A)$ generated by
$$u_1=v_1+v_3,\quad\quad u_3=v_2+v_4,$$
$$u_2=\frac{(1-F)u_1-(b-2+\alpha)u_1-\ell^{m_b}u_3}{\ell^{m_1}},\text{ and }$$
$$u_4=\frac{(1-F)u_3-\alpha
u_3+\frac{f(1)}{\ell^{m_b}}u_1}{\ell^{m_2}}.$$ 

A straightforward calculation shows that

$$(1-F)u_1=\ell^{m_1}u_2+(b-2+\alpha)u_1+\ell^{m_b}u_3;$$
$$(1-F)u_3=\ell^{m_2}u_4+\alpha u_3-\frac{f(1)}{\ell^{m_b}}u_1;$$
$$(1-F)u_2=\frac{-(b-2)\alpha u_1-\ell^{m_b}\alpha u_3}{\ell^{m_1}};$$
$$(1-F)u_4=\frac{f(1)\alpha u_1}{\ell^{m_2+m_b}}\in\ell^{m_4}T_\ell.$$

By $(a)$ and $(1)$, the vector $(1-F)u_1\in\ell^{m_1}T_\ell$;
by $(3)$ and $(4)$, the vector $(1-F)u_3\in\ell^{m_2}T_\ell$;
by $(a)$ the vector $(1-F)u_2\in\ell^{m_3}T_\ell$;
by definition of $m_b$, the vector $(1-F)u_4\in\ell^{m_4}T_\ell$.
We get that the natural surjective map $T_\ell\to G_\ell$ factors through
a surjective map $T_\ell/(1-F)T_\ell\to G_\ell$. We conclude that
$T_\ell/(1-F)T_\ell\cong G_\ell$, since orders of both groups coincide.

\end{document}